\documentclass[11pt]{amsart}
\usepackage[dvipsnames,usenames]{color}
\usepackage[colorlinks=true, urlcolor=NavyBlue, linkcolor=NavyBlue, citecolor=NavyBlue]{hyperref}
\usepackage{graphicx}
\usepackage{epsfig}
\usepackage[latin1]{inputenc}
\usepackage{amsmath}
\usepackage{amsfonts}
\usepackage{amssymb}
\usepackage{amsthm}
\usepackage{amscd}
\usepackage{verbatim}
\usepackage{subfigure}
\usepackage{pinlabel}
\usepackage{stmaryrd}
\usepackage{enumerate, enumitem}

\usepackage{tikz}
\usetikzlibrary{arrows}
\usetikzlibrary{decorations.pathreplacing}
\usepackage{verbatim}
	
\newtheorem{theorem}{Theorem}[section]
\newtheorem{Theorem}{Theorem}
\newtheorem{lemma}[theorem]{Lemma}

\newtheorem{Corollary}[Theorem]{Corollary}

\theoremstyle{definition}

\theoremstyle{remark}
\newtheorem{remark}[theorem]{Remark}

\newtheorem*{ack}{Acknowledgements}

\def\Z{\mathbb{Z}}

\def\R{\mathbb{R}}

\def\F{\mathbb{F}}

\def\CFKi{CFK^{\infty}}

\def\varep{\varepsilon}

\def\varep{\varepsilon}

\title{A note on the concordance invariants epsilon and upsilon}

\subjclass[2013]{}
\author[Jennifer Hom]{Jennifer Hom}
\thanks{The author was partially supported by NSF grant DMS-1307879.}
\address {Department of Mathematics, Columbia University, 2990 Broadway \\ New York, NY 10027}
\email{hom@math.columbia.edu}

\numberwithin{equation}{section}

\begin{document}

\begin{abstract}
Ozsv\'ath-Stipsicz-Szab\'o \cite{OSS} recently defined a one-parameter family $\Upsilon_K(t)$ of concordance invariants associated to the knot Floer complex. We compare their invariant to the $\{ -1, 0, 1\}$-valued concordance invariant $\varepsilon(K)$, which is also associated to the knot Floer complex. In particular, we give an example of a knot $K$ with $\Upsilon_K(t) \equiv 0$ but $\varepsilon(K) \neq 0$. 
\end{abstract}

\maketitle

\vspace{-7pt}
\section{Introduction}

Beginning with the $\Z$-valued concordance homomorphism $\tau(K)$ \cite{OS4ball}, the knot Floer homology package \cite{OSknots, R} has yielded an abundance of concordance invariants. One of the benefits of these invariants, as opposed to classical concordance invariants such as signature, is that they can be non-vanishing on topologically slice knots. For example, we have the following theorem.

\begin{Theorem}[{\cite[Theorem 1]{Homsummand}}]
\label{thm:summand}
The subgroup of the smooth concordance group given by topologically slice knots contains a direct summand isomorphic to $\Z^\infty$.
\end{Theorem}

\noindent The proof of the above theorem relies on the $\{-1, 0, 1\}$-valued concordance invariant $\varep(K)$ associated to the knot Floer complex \cite[Definition 3.1]{Homsmooth}. The quotient of the concordance group by the subgroup $\{ K \mid \varep(K)=0 \}$ is totally ordered, and properties of the order structure can be used to construct linearly independent concordance homomorphisms.

Ozs\'vath-Stipsicz-Szab\'o \cite[Theorem 1.20]{OSS} recently gave a new proof of Theorem \ref{thm:summand}, using a one-parameter family $\Upsilon_K(t)$ of $\R$-valued concordance homomorphisms also associated to the knot Floer complex. Both $\varep$ and $\Upsilon$ are strictly stronger than $\tau$ in that
\[ \varep(K) = 0 \textup{ implies } \tau(K) =0 \qquad \textup{ and } \qquad \Upsilon_K(t) \equiv 0 \textup{ implies } \tau(K) =0, \]
but there exist knots $K$ with $\tau(K) = 0$ while $\varep(K) \neq 0$ and $\Upsilon_K(t) \not\equiv 0$. One such example is the knot $T_{3,4} \# -T_{2, 7}$, where $T_{p, q}$ denotes the $(p, q)$-torus knot and $-K$ denotes the reverse of the mirror image of $K$.

The knot Floer complex $\CFKi(K)$ is a bifiltered chain complex associated to the knot $K$. We call the two filtrations the vertical and horizontal filtrations. The invariants $\varep$ and $\Upsilon$ are both defined using the bifiltration, while the definition of $\tau$ uses only one of the two filtrations. Roughly, $\varep(K)$ is a measure of how the vertical filtration interacts with the horizontal filtration: the so-called vertical homology has rank one, and $\varep$ measures whether this homology class is a boundary, cycle, or neither in the horizontal homology. On the other hand, the idea behind $\Upsilon_K(t)$ is to apply a linear transformation to the bifiltration on the knot Floer complex and then look at the grading of a certain distinguished generator in the homology of the resulting complex.

More generally, both $\varepsilon$ and $\Upsilon$ are invariants of not just knots, but  of (suitable) bifiltered chain complexes.  In \cite[Proposition 9.4]{OSS}, Ozsv\'ath-Stipsicz-Szab\'o give an example of a complex $C$ with $\varepsilon(C) = 0$ but $\Upsilon_C(t) \not\equiv 0$, although it is currently unknown if the complex $C$ is realized as $\CFKi$ of a knot. Conversely, we prove the following.

\begin{Theorem}
\label{thm:main}
There exist knots $K$ with $\Upsilon_K(t) \equiv 0$ but $\varepsilon(K) \neq 0$.
\end{Theorem}

\noindent The knots used in the above theorem are connected sums of certain (iterated) torus knots.

An interesting question to consider is what obstructions to sliceness can be extracted from $\CFKi(K)$ when $\Upsilon_K(t) \equiv 0$ and $\varep(K) = 0$.

Recall that the \emph{concordance genus} of $K$, $g_c(K)$, is the minimal Seifert genus of any knot $K'$ which is concordant to $K$. The function $\Upsilon_K(t)$ is a piecewise-linear function of $t$ whose slope has finitely many discontinuities \cite[Proposition 1.4]{OSS}. Let $s$ denote the maximum of the finitely many slopes appearing in the graph of $\Upsilon_K(t)$. Ozsv\'ath-Stipsicz-Szab\'o \cite[Theorem 1.13]{OSS} prove that
\[ s \leq g_c(K). \]
There is also a concordance genus bound $\gamma(K)$, defined using $\varepsilon$ \cite{Homconcordancegenus}. 

\begin{Corollary}
\label{cor:concordancegenus}
There exist knots $K$ for which the concordance genus bound given by $\Upsilon_K(t)$ is zero, but $\gamma(K) \neq 0$.
\end{Corollary}

\begin{ack}
I would like to thank Peter Ozsv\'ath for useful correspondence, and Tye Lidman for helpful comments on an earlier draft. 
\end{ack}

\section{The example}

We will let $T_{p, q; s, t}$ denote the $(s, t)$-cable of $T_{p,q}$, where $s$ denotes the longitudinal winding. We assume the reader is familiar with the knot Floer complex; see, for example, \cite[Section 2]{Homsmooth} and \cite[Section 2]{OSS}.

\begin{lemma}
\label{lem:summandK}
Let $K = T_{4,5} \# -T_{2,3; 2,5}$. Then $\CFKi(K)$ contains a direct summand generated over $\F[U, U^{-1}]$ by $x, y,$ and $z$ with
\begin{align*}
	M(x) & = 0  &A(x) & = 2 \\
	M(y) & = -3  &A(y) & = 0 \\
	M(z) & = -4  &A(z) & = -2
\end{align*}
and differential
\begin{align*}
	\partial x = 0 \qquad \qquad \partial y = U^2 x + z \qquad \qquad \partial z =0.
\end{align*}
Here, $M$ and $A$ denote the Maslov grading and Alexander filtration, respectively.
\end{lemma}

\begin{proof}
The knot $T_{2,3; 2,5}$ is an L-space knot \cite[Theorem 1.10]{HeddencablingII}; see also \cite{HomLspace}. The Alexander polynomial of $T_{2,3; 2,5}$ is
\begin{align*}
	 \Delta_{T_{2,3; 2,5}}(t) &= \Delta_{T_{2, 3}}(t^2) \cdot \Delta_{T_{2,5}} \\
	 &= t^4 - t^3 + 1 -t^{-3} + t^{-4}.
\end{align*}
Then by \cite{OSlens} (as restated in \cite[Theorem 2.10]{OSS}), the complex $\CFKi(T_{2,3; 2,5})$ is generated over $\F[U, U^{-1}]$ by $a, b, c, d,$ and $e$ with
\begin{align*}
	M(a) & = 0  &A(a) & = 4 \\
	M(b) & = -1  &A(b) & = 3 \\
	M(c) & = -2  &A(c) & = 0 \\
	M(b) & = -7  &A(b) & = -3 \\
	M(c) & = -8  &A(c) & = -4
\end{align*}
and differential
\begin{align*}
	\partial a = \partial c = \partial e =0 \qquad \qquad \partial b = U a + c \qquad \qquad \partial d = U^3 c + e.
\end{align*}
In the language of  \cite[Section 2.4]{HancockHomNewman}, we have that $\CFKi(T_{2,3; 2,5})$ can be denoted $[1, 3]$, and the summand $C$ specified in the statement of Lemma \ref{lem:summandK} can be denoted $[2]$. This notation refers to the lengths of the horizontal and vertical arrows in a graphical depiction of $\CFKi$, beginning from the generator of vertical homology and continuing to the point of symmetry. See Figures \ref{subfig:T_2325}  and \ref{subfig:K}. It then follows from \cite[Lemma 3.1]{HancockHomNewman} that we have that $\CFKi(T_{2,3; 2,5}) \otimes C$ is of the form $[1, 3, 2]$.

The Alexander polynomial of $T_{4, 5}$ is
\[ \Delta_{T_{4,5}}(t) = t^6 - t^5 + t^2 - 1 + t^{-2} -t^{-5} + t^6.\]
Since $T_{4,5}$ admits a lens space surgery, it an L-space knot. Thus, we may apply \cite[Theorem 2.10]{OSS} to obtain a description of $\CFKi(T_{4,5})$, and we see that, in the notation of \cite[Section 2.4]{HancockHomNewman}, this complex is of the form $[1, 3, 2]$. See Figure \ref{subfig:T_45}.

\begin{figure}[htb!]
\vspace{5pt}
\labellist
\small \hair 2pt
\endlabellist
\centering
\subfigure[]{
\begin{tikzpicture}[scale=0.7]
	\begin{scope}[thin, gray]
		\draw [<->] (-1, 0) -- (5, 0);
		\draw [<->] (0, -1) -- (0, 7);
	\end{scope}
	\draw[step=1, black!30!white, very thin] (-0.9, -0.9) grid (4.9, 6.9);
	\filldraw (0, 4) circle (2pt) node[] (a){};
	\filldraw (1, 4) circle (2pt) node[] (b){};
	\filldraw (1, 1) circle (2pt) node[] (c){};
	\filldraw (4, 1) circle (2pt) node[] (d){};
	\filldraw (4, 0) circle (2pt) node[] (e){};
	\draw [very thick, <-] (a) -- (b);
	\draw [very thick, <-] (c) -- (b);
	\draw [very thick, <-] (c) -- (d);
	\draw [very thick, <-] (e) -- (d);
\end{tikzpicture}
\label{subfig:T_2325}
}
\hspace{25pt}
\subfigure[]{
\begin{tikzpicture}[scale=0.7]
	\begin{scope}[thin, gray]
		\draw [<->] (-1, 0) -- (3, 0);
		\draw [<->] (0, -1) -- (0, 7);
	\end{scope}
	\draw[step=1, black!30!white, very thin] (-0.9, -0.9) grid (2.9, 6.9);
	\filldraw (0, 2) circle (2pt) node[] (a){};
	\filldraw (2, 2) circle (2pt) node[] (b){};
	\filldraw (2, 0) circle (2pt) node[] (c){};
	\draw [very thick, <-] (a) -- (b);
	\draw [very thick, <-] (c) -- (b);
\end{tikzpicture}
\label{subfig:K}
}
\hspace{25pt}
\subfigure[]{
\begin{tikzpicture}[scale=0.7]
	\begin{scope}[thin, gray]
		\draw [<->] (-1, 0) -- (7, 0);
		\draw [<->] (0, -1) -- (0, 7);
	\end{scope}
	\draw[step=1, black!30!white, very thin] (-0.9, -0.9) grid (6.9, 6.9);
	\filldraw (0, 6) circle (2pt) node[] (a){};
	\filldraw (1, 6) circle (2pt) node[] (b){};
	\filldraw (1, 3) circle (2pt) node[] (c){};
	\filldraw (3, 3) circle (2pt) node[] (d){};
	\filldraw (3, 1) circle (2pt) node[] (e){};
	\filldraw (6, 1) circle (2pt) node[] (f){};
	\filldraw (6, 0) circle (2pt) node[] (g){};
	\draw [very thick, <-] (a) -- (b);
	\draw [very thick, <-] (c) -- (b);
	\draw [very thick, <-] (c) -- (d);
	\draw [very thick, <-] (e) -- (d);
	\draw [very thick, <-] (e) -- (f);
	\draw [very thick, <-] (g) -- (f);
\end{tikzpicture}
\label{subfig:T_45}
}
\caption{Left, $\CFKi(T_{2,3; 2,5})$. Center, the relevant summand of $\CFKi(T_{4,5} \#-T_{2,3; 2,5})$ from the statement of Lemma \ref{lem:summandK}. Right, $\CFKi(T_{4,5})$. More precisely, $\CFKi$ is generated over $\F[U, U^{-1}]$ by the generators depicted.}
\label{fig:CFK}
\end{figure}
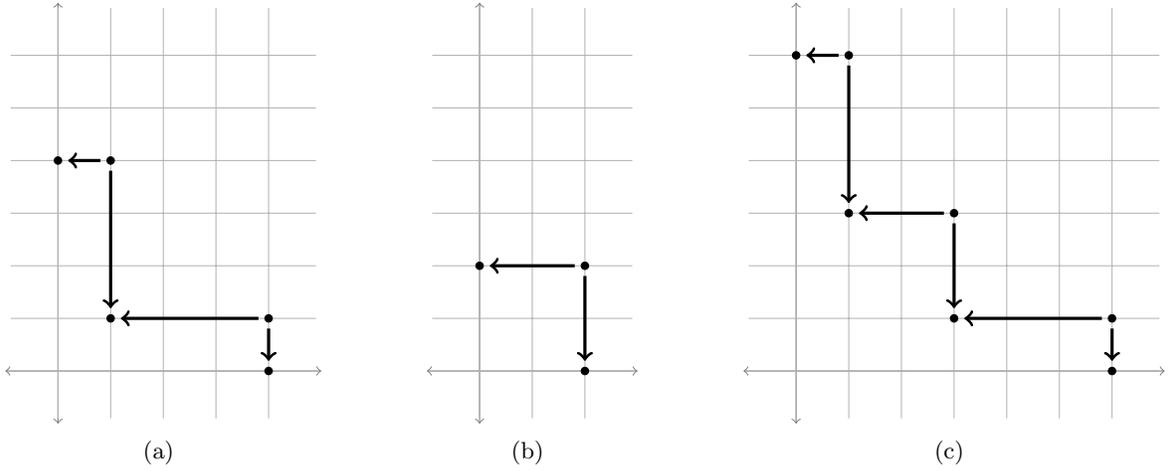

It follows from \cite[Section 2.4]{HancockHomNewman} that since $\CFKi(T_{2,3; 2,5}) \otimes C$ has the same form as $\CFKi(T_{4,5})$, the complex $C$ is a direct summand of $\CFKi(T_{4,5}) \otimes \CFKi(T_{2,3; 2,5})^*$, or, equivalently, $\CFKi(T_{4,5} \#-T_{2,3; 2,5})$.
\end{proof}

\begin{lemma}
\label{lem:UpsilonK}
Let $K = T_{4,5} \# -T_{2,3; 2,5}$. Then 
\[ \Upsilon_K(t) = \left\{
	\begin{array}{ll}
		-2t & \textup{if } 0 \leq t \leq 1\\
		2t-4 & \textup{if } 1 < t \leq 2.
	\end{array} \right. \]
\end{lemma}

\begin{proof}
The summand of $\CFKi(K)$ described in Lemma \ref{lem:summandK} generates the homology of the total complex $\CFKi(K)$. In particular, this summand determines $\Upsilon_K(t)$. Although this summand is not itself $\CFKi$ of an L-space knot \cite[Corollary 9]{HeddenWatson}, the calculation in \cite[Proof of Theorem 6.2]{OSS} still applies, yielding the desired result.
\end{proof}

\begin{lemma}
\label{lem:UpsilonT25}
For the $(2,5)$-torus knot, we have 
\[ \Upsilon_{T_{2,5}}(t) = \left\{
	\begin{array}{ll}
		-2t & \textup{if } 0 \leq t \leq 1\\
		2t-4 & \textup{if } 1 < t \leq 2.
	\end{array} \right. \]
\end{lemma}

\begin{proof}
The result follows immediately from \cite[Theorem 1.15]{OSS}.
\end{proof}

With these lemmas in place, we are now ready to prove Theorem \ref{thm:main}.

\begin{proof}[Proof of Theorem \ref{thm:main}]
By \cite[Propositions 1.8 and 1.9]{OSS},
\[
	\Upsilon_{K_1 \# K_2} (t) = \Upsilon_{K_1} (t) + \Upsilon_{K_2} (t) \qquad \textup{ and }\qquad \Upsilon_{-K}(t) = -\Upsilon_K(t).
\]
Combined with Lemmas \ref{lem:UpsilonK} and \ref{lem:UpsilonT25}, it follows that
\[ \Upsilon_{T_{2, 5} \# -T_{4,5} \# T_{2,3; 2,5}}(t) \equiv 0. \]

We consider the invariant $a_1(K)$ defined in \cite[Section 6]{Homsmooth}. For complexes such as the ones in Figure \ref{fig:CFK}, the invariant $a_1(K)$ is equal to the length of the horizontal arrow coming in to the generator of vertical homology. From the partial description of $\CFKi(T_{4,5} \# -T_{2,3; 2,5})$ in Lemma \ref{lem:summandK}, it follows that
\[ a_1(T_{4,5} \# -T_{2,3; 2,5}) = 2. \]
By \cite[Lemma 6.5]{Homsmooth} we have that
\[ a_1(T_{2,5}) = 1. \]
Lastly, by \cite[Lemma 6.3]{Homsmooth} we have that if $a_1(J) > a_1(K)$, then $\varep(K \# -J) =1$. Thus
\[ \varep(T_{2, 5} \# -T_{4,5} \# T_{2,3; 2,5}) =1, \]
as desired. 

Recall from \cite[Proposition 3.6]{Homcables} that for $n>0$, we have
\[ \varep(nK) = \varep(K) \quad \textup{ and } \quad \varep(-K) = -\varep(K), \]
It follows that any non-zero multiple $nK$ of the knot $K=T_{2, 5} \# -T_{4,5} \# T_{2,3; 2,5}$ will also have the property that $\Upsilon_{nK}(t) \equiv 0$ and $\varep(nK) \neq 0$.
\end{proof}

\begin{proof}[Proof of Corollary \ref{cor:concordancegenus}]
The invariant $\gamma(K)$ vanishes if and only if $\varep(K)=0$. Hence $K=T_{2, 5} \# -T_{4,5} \# T_{2,3; 2,5}$ (or any non-zero multiple thereof) has the desired property.
\end{proof}

\begin{remark}
Let $K=T_{2, 5} \# -T_{4,5} \# T_{2,3; 2,5}$. By computing $\CFKi(K)$ using the K\"unneth formula \cite[Theorem 7.1]{OSknots}, one can determine that $\gamma(K)=4$. More generally, we expect that $\gamma(nK)=4n$, giving knots for which the concordance genus bound obtained from $\Upsilon_K(t)$ is zero, but the bound obtained from $\gamma$ is arbitrarily large.
\end{remark}

\bibliographystyle{amsalpha}

\bibliography{mybib}

\end{document}